\documentclass[[a4paper,10pt]{article}
\usepackage[english]{babel}
\usepackage[utf8]{inputenc}
\usepackage[T1]{fontenc}
\usepackage{times}

\usepackage{hyperref}
\usepackage[mathcal]{eucal}
\usepackage{multicol}
\usepackage{graphicx}
\usepackage{amsthm,latexsym,amsfonts,amsmath,amssymb}
\usepackage[top=2cm, bottom=1cm, left=2cm, right=1cm]{geometry}
\usepackage{mathrsfs,fancybox}
\usepackage{mathrsfs,latexsym,amsthm,amsfonts,amsmath,amssymb}

\newcommand{\X}{\mathbb{X}}
\newcommand{\Y}{\mathbb{Y}}
\newcommand{\R}{\mathbb{R}}

\newtheorem{theorem}{Theorem}[section]
\newtheorem{prop}[theorem]{Proposition}
\newtheorem{cor}[theorem]{Corollary}
\newtheorem{lemma}[theorem]{Lemme}
\newtheorem{remark}[theorem]{Remark}
\newtheorem{definition}[theorem]{Definition}
\title{Dynamics and oscillations of a predator–prey model with modified Leslie–Gower
Holling-type II schemes time-dependent delays}
\author{ \textsl{\ Haifa Ben Fredj\footnote{MaPSFA, ESSTHS, University of Sousse-Tunisia.  
{\bf{Email:}} haifabenfredjh@gmail.com}} \hspace{1cm}
 \and
\textsl{ Farouk Chérif\footnote{ISSATs and MaPSFA (LR11ES35), ESSTHS, University of Sousse-Tunisia. 
{\bf{Email:}} farouk.y.cherif@gmail.com}} 
}
\begin{document}
\maketitle
\thispagestyle{empty}

\begin{abstract}
A predator–prey system is investigated in this research, which is based on a modified version of the Leslie–Gower scheme and a Holling-type II scheme with time-dependent delays. Using Schauder's fixed-point theorem,
we studied the existence  of pseudo almost periodic  solution for the suggested model.  Based on the   suitable
Lyapunov functional, sufficient conditions are established for the  globally attractive pseudo almost periodic solution. At the end, two numerical examples are presented to demonstrate the effectiveness of our results.
\end{abstract}
\textbf{keywords:} Prey-predator model, Pseudo almost periodic solution, continuous delays, Global exponential stability.

\section{Introduction}
One of the most attractive and significant issues in mathematical ecology is the connection between prey and predator. The predator–prey interaction is heavily influenced by functional responses. As a result, both mathematicians and ecologists have looked at predator–prey systems with a wide variety of functional responses.
In particulier, Aziz-Alaoui and Daher \cite{aziz2003boundedness} designed and investigated the following predator–prey model with modified Leslie–Gower and Holling-type II:
\begin{equation}\label{Aloui}
\left\{\begin{aligned}
&\dfrac{du}{dt}=\left(a_1-bu-\dfrac{c_1v}{k_1+u}\right)u,\\
&\dfrac{dv}{dt}=\left(a_2-\dfrac{c_2v}{k_2+u}\right)v,
\end{aligned}
\right.
\end{equation}
with initial value $u(0) = u_0 > 0$, $v(0) = v_0 > 0$, where $u(t)$ and $v(t)$ stand for the prey population size and the predator population size respectively, and $a_1$, $a_2$,  $b$, $c_1$, $c_2$, $k_1$ and $k_2$ are all positives with the ecology meaning as follows

\begin{table}[hbtp]
\begin{center}
\begin{tabular}{p{0.7cm}p{12.7cm}}

$a_1$&the growth rate of prey\\
$a_2$&the growth rate of predator\\
$b$&measures the strength of competition among individuals of species $u$\\
$c_1$&is the maximum value which per capita reduction rate of $u$ can attain\\
$c_2$&is the maximum value which per capita reduction rate of $v$ can attain\\
$k_1$&measures the extent to which environment provides protection to prey $u$\\
$k_2$&measures the extent to which environment provides protection to predator $v$
\end{tabular}
\end{center}
\end{table}
This model can treat the interaction prey-predator which based on the following assumptions
\begin{enumerate}
\item the prey growth following the logistic equation (i.e $u'(t)=(a_1 -bu(t))u(t)$ ) in the absence of his predator.
\item the link between the attack rate and predator size describe following the Holling type II $\left(\frac{c_1v}{k_1+u}\right)$ which a predator's rate of prey consumption grows as prey density grows, but finally reaches a plateau (or asymptote) where the rate of consumption remains constant regardless of prey density increases.
\item the Leslie–Gower formula is built on the premise that a predator population's decrease is proportional to the availability of its favorite food per capita. It is $\frac{dv}{dt}=a_2v\left(1-\frac{v}{\alpha u}\right)$, in which the growth of the predator population is of logistic form i.e. $\frac{dv}{dt}=a_2v\left(1-\frac{v}{C}\right)$
.
Here, "$C$" measures the carry capacity set by the environmental resources and is proportional to prey abundance, $C=\alpha u$, 
where is the conversion factor of prey into predator \cite{leslie1948some,leslie1960properties,steinmuller1978pielou}.
In the case of severe scarcity, $v$ can switch over to other populations but its growth will be limited
by the fact that its most favorite food $u$ is not available in abundance. This case can be taken care of by adding a
positive constant k to the denominator, see \cite{ aziz2002study,aziz2003boundedness}.
\end{enumerate}

Many natural and man-made processes in biology, medicine, and other fields now incorporate time-delays, according to current research. Time delays occur so often in nearly every circumstance that ignoring them is ignoring reality. Kuang \cite{kuang1993global} mentioned that animals must take time to digest their food before further activities and responses take place, and hence any model of species dynamics without delays is an approximation at best. It is now beyond question that the influence of time-delay owing to the time necessary to transition from egg to adult stage, gestation duration, and other factors must be considered in a better study. The famous works of Macdonald \cite{macdonald2008biological}, Gopalsamy \cite{gopalsamy1992stability}, and Kuang \cite{kuang1993global} include detailed arguments for the relevance and use of time-delays in realistic models. As a result, the "Ordinary Differential Equation", which is at the heart of Mathematical Biology, should be replaced by the "Delay Differential Equation".\\

Furthermore, the occurrence of almost periodic solutions is one of the most fascinating subjects in qualitative differential equations since they may be used to dynamic of prey-predator system \cite{liao2018almost,menouer2017existence,tripathi2020almost,xu2019global}. An extension of the almost periodic function is the pseudo almost  periodic function. It was defined in \cite{zhang1994pseudo}. 
It is worth noting that due to their potential applicability in a wide range of fields, almost periodic and pseudo almost periodic solutions have received a lot of attention in the last decade \cite{amdouni2018pseudo,bekolle2021attractiveness,chen2016positive}.\\

Roughly speaking, we shall consider  the following differential system of predator-prey model which incorporates the Holling type II and a modified Leslie-Gower functional response:
\begin{equation}\label{prey-predatorII}
\left \{\begin{aligned}
u'(t)&=\left(a_1(t)-b(t)u(t)-\dfrac{c_1(t)v(t-\tau_1(t))}{u(t-\sigma_1(t))+k_1(t)}\right)u(t);\\
v'(t)&=\left(a_2(t)-\dfrac{c_2(t)v(t-\tau_2(t))}{u(t-\sigma_2(t))+k_2(t)}\right)v(t), 
\end{aligned}
\right. 
\end{equation}
where $a_i,b,c_i,k_i,\tau_i,\sigma_i:[0,+\infty[\rightarrow ]0,+\infty[,$ $i=1,2$, are continuous functions. The term {\small$\dfrac{c_2(t)v(t-\tau_2(t))}{u(t-\sigma_2(t))+k_2(t)}$} is of this equation is called the Leslie-Gower
term and the term {\small$\dfrac{c_1(t)u(t)}{u(t-\sigma_1(t))+k_1(t)}$ } is the Holling II functional response. Pose $$r =\underset{t \in \R}{sup} \left(\tau_i(t),\sigma_i(t) ; i=1,2\right).$$

 Denote by $BC ([-r, 0] , \R^2_+$) the set of bounded continuous functions from $[-r, 0]$ to $\R^2_+$. If $z(.)$ is defined on $[-r + t_0, \rho[$ with $t_0, \rho \in \R$, then we define $z_t \in C([-r, 0] , \R^2$) where $z_t(\theta) = z(t + \theta)$ for all $\theta \in [-r, 0]$. Notice that we restrict our selves to $\R^2_+$-valued functions since only non-negative solutions of \eqref{prey-predatorII} are biologically meaningful. So, let us consider the following initial condition
\begin{equation}\label{condition}
 z_{t_0} = \phi,\quad \phi=(\phi_1,\phi_2) \in BC ([-r, 0] , \R^2_+) \text{ and }\phi_1 (0),\phi_2(0) > 0. 
 \end{equation}
We write $z_t (t_0, \phi)$ for a solution of the admissible initial value problem \eqref{prey-predatorII} and \eqref{condition}. Also, let $[t_0, \eta(\phi)[$ be the maximal right-interval of existence of $z_t(t_0, \phi)$.
\section{Preliminaries and definitions}
Throughout this paper, for all functions $f\in BC(\R,\R)$, we note :
$$ f^s=\underset{x\in \R}{sup}|f|\text{ and }\quad f^i=\underset{x\in\R}{inf}|f|.$$

\begin{definition}\cite{chuanyi2003almost}\item
Let $f\in BC(\R,\X)$. $f$ is said almost periodic (a.p) if  for any $ \epsilon>0$, there exists $l_\epsilon >0$, such that
$$  \exists \tau \in [a,a+l_\epsilon[, \forall a\in \R, \quad ||f(x+\tau)-f(x)||_{\X}<\epsilon,$$
As well know $\tau$ is called $\epsilon-period$ of $f$. We denote by $AP(\R,\X))$  the set of such functions.
\end{definition}
 It is well known that the set $AP(\R,\X)$ is a Banach space with the supremum norm:
 $$||f||_{\infty}=\underset{x\in \R}{sup}||f(x)||_{\X}.$$

In the early 1990's, the concept of pseudo almost periodicity (p.a.p) was introduced by Zhang (see \cite{chuanyi2003almost}). It is a  generalization of the almost periodicity.  Define the class of functions $PAP_0 (\R,\X))$ as follows:$$ PAP_0 (\R,\X))=\bigg\{ f \in BC(\R,\X));
\underset{T\rightarrow +\infty }{lim} \int^T_{-T}||f(t)||_{\X} dt=0\bigg\}.$$

\begin{definition}\cite{chuanyi2003almost}\item A function $f \in BC\R,\X))$ is called pseudo almost-periodic if it can be expressed as 
$$f =f_1+f_2, \quad where \quad f_1,f_2\in AP(\R,\X))\times PAP_0(\R,\X)). $$
\end{definition}
 \begin{prop}\cite{chuanyi2003almost} 
1. $(PAP(\R,\X)),||.||_{\infty})$ is Banach space and 
 $$AP(\R,\X))\varsubsetneq PAP(\R,\X))\varsubsetneq BC(\R,\X)).$$
2.For $f\in PAP(\R,\X))$ and $g\in PAP(\R,\R)$ with $inf_{t\in \R}|g(t)|>0$, then $\dfrac{f}{g}\in PAP(\R,\X)).$
 \end{prop}

\begin{definition} (Definition 2.12, \cite{chuanyi2003almost}) \item  Let $\Omega \subseteq Y$. An continuous function f : $\R \times\Omega \longrightarrow \X$  is called pseudo almost periodic (p.a.p).
in t uniformly with respect $x \in \Omega$ if the two following conditions are satisfied :\\
i) $\forall x \in  \Omega$, $f(., x) \in PAP(\R,\X)$,\\
ii) for all compact $K$ of $\Omega$, $\forall \epsilon> 0, \exists \delta > 0, \forall t \in \mathbb{R}, \forall x_1, x_2 \in K$,
\begin{center}
$||x_1 - x_2||_{\Y} \leq \delta \Rightarrow  ||f(t, x_1) - f(t, x_2)||_{\X} \leq \epsilon$.
\end{center}
Denote by $PAP_U(\R\times \Omega; \X)$  the set of all such functions.
\end{definition}
\section{Positivity and Bounded of the  solution}

\begin{theorem}\label{positive}
Let $(u,v) \in \R^2$ solution of system  \eqref{prey-predatorII}. If the the initial condition \eqref{condition} is satisfied, then the solution $(u,v)$ is strictly positive. 
\end{theorem}

\begin{proof}
By integration from $t_0$ into t of the system \eqref{prey-predatorII}, we have
\begin{align*}
\displaystyle u(t)&=\phi_1(0)exp\left( \int_{t_0}^t a_1(s)-b(s)u(s)-\frac{c_1(s)v(s-\tau_1(s))}{u(s-\sigma_1(s))+k_1(s)}ds\right) ,\\
\displaystyle v(t)&=\phi_2(0)exp\left(\int_{t_0}^t a_2(s)-\frac{c_2(s)v(s-\tau_2(s))}{u(s-\sigma_2(s))+k_2(s)}ds\right).
\end{align*}
then it is clear that the solution $(u,v)$ has the same sign as the initial condition \eqref{condition}. Hence, the solution is strictly positive. 
\end{proof}
\begin{definition}\cite{lu2015permanence}\item
We will say that  the solution $(u,v)$ of \eqref{prey-predatorII} is :\begin{enumerate}
\item permanent if there
exists $M_1,M_2 > 0$ such that 
$$0 < \liminf_{t\rightarrow +\infty}
u(t) \leq \limsup_{t\rightarrow +\infty}
u(t) \leq M_1\quad\& \quad 0 < \liminf_{t\rightarrow +\infty}
v(t) \leq \limsup_{t\rightarrow +\infty}
v(t) \leq M_2$$
\item is  uniformly permanent if there exists $M_1 > m_1> 0$ and $M_2 > m_2> 0$ such that
$$m_1 \geq \liminf_{t\rightarrow +\infty}
u(t) \leq \limsup_{t\rightarrow +\infty}
u(t) \leq M_1\quad\& \quad m_2 < \liminf_{t\rightarrow +\infty}
v(t) \leq \limsup_{t\rightarrow +\infty}
v(t) \leq M_2$$

\end{enumerate}
\end{definition}

\begin{lemma}\label{lem1} \cite{chen2007note}	Let $a>0,b>0$.
\begin{enumerate}
\item[1.] If $dx(t)/dt\geq x(b-ax)$, then $\lim  inf_{t\rightarrow +\infty}x(t)\geq b/a$ for $t\geq 0$ and $x(0)>0;$
\item[2.] If $dx(t)/dt\leq x(b-ax)$, then $\lim sup_{t\rightarrow +\infty}x(t)\leq b/a$ for $t\geq 0$ and $x(0)>0.$
\end{enumerate} 
\end{lemma}

\begin{theorem}\label{uniformpermanent}[Uniform permanent]
\item If 
$$(C0)\qquad a_1^ik_1^i-M_2 c_1^s>0$$
holds, then  any  positive solution $(u(t),v(t))^T$ of differential system \eqref{prey-predatorII} satisfies 
$$\begin{aligned}
&m_1\leq\text{ } \lim inf_{t\rightarrow +\infty} u(t)\leq\text{  }\lim  sup_{t\rightarrow +\infty} u(t)\leq  M_1  ;\\
&m_2 \leq \text{ } \lim inf_{t\rightarrow +\infty} v(t)\leq \text{  } \lim sup_{t\rightarrow +\infty} v(t)\leq M_2.
\end{aligned}$$ 
where \[\left\{\begin{aligned}
 M_1:=\dfrac{a_1^s}{b^i}, \quad & m_1:=\dfrac{a_1^ik_1^i-M_2  c_1^s}{b^s k^i_1}\\
 M_2:=\dfrac{a^s_2(M_1+k_2^s) exp(a^s_2 \tau_2^s)}{c_2^i},\quad & m_2:=\dfrac{a_2^i(m_1+k_2^i)}{c_2^sexp\bigg(\dfrac{c_2^s M_2\tau_2^s}{k_2^i+m_1}\bigg)}
\end{aligned}
\right.\]
\end{theorem}

\begin{proof}
 
Let $(u(t),v(t))^T$ be any positive solution of system \eqref{prey-predatorII}. It follows from the first equation system \eqref{prey-predatorII} that
\begin{equation}\label{du}
\dfrac{du(t)}{dt}\leq u(t)\big [a_1^s-b^iu(t)\big]; 
\end{equation}
From Lemma \ref{lem1}, we get
\begin{equation}\label{M1}
\limsup_{t\rightarrow +\infty} u(t)\leq \dfrac{a_1^s}{b^i}:=M_1.
\end{equation}
From the second equation in system \eqref{prey-predatorII}, we have
\begin{equation}\label{dv}
\dfrac{dv(t)}{dt}\leq a_2^s v(t).
\end{equation}
By integrating \eqref{dv} from $t-\tau_2(t)$ to $t$, we have
\begin{equation}\label{v}
v(t)\leq v(t-\tau_2(t))e^{a_2^s \tau_2(t)}\leq v(t-\tau_2(t))e^{a_2^s \tau_2^s}.
\end{equation}
Then,
\begin{equation}\label{v(tau)}
v(t-\tau_2(t))\geq v(t) exp(-a_2^s \tau_2^s).
\end{equation}
By \eqref{v(tau)} and the second equation of system \eqref{prey-predatorII}
\begin{equation}\label{d2v} 
\dfrac{dv(t)}{dt}\leq  v(t)\bigg[a_2^s-\dfrac{c_2^iexp\left(-a_2^s \tau_2^s\right)}{M_1+k_2^s} v(t) \bigg].
\end{equation}
From lemma \ref{lem1}, we get 
\begin{equation}\label{M2}
v(t)\leq \dfrac{a^s_2(M_1+k_2^s)}{c_2^ie^{-a_2^s \tau_2^s}}:=M_2.  
\end{equation}
By positivity of the solution $(u,v)$ and from \eqref{M2}, we have
\begin{equation}\label{d2u}
\dfrac{du(t)}{dt}\geq u(t)\big [a_1^i-\dfrac{M_2c_1^s}{k_1^i}-b_1^su(t)\big].\end{equation}
From lemma \ref{lem1}, we obtain
$$\liminf_{t\rightarrow +\infty}u(t)\geq \dfrac{a_1^ik_1^i-M_2 c_1^s}{b_1^s k^i_1}:=m_1.$$

By integrating  second inequality of system \eqref{prey-predatorII} from $t$ to $t-\tau_2(t)$, we have 
\begin{equation}
v(t)\geq v(t-\tau_2(t))exp\bigg(-\dfrac{c_2^s M_2}{k_2^i+m_1}\tau_2(t)\bigg).
\end{equation}
Then,
\begin{equation}
 v(t-\tau_2(t))\leq v(t)exp\bigg(\dfrac{c_2^s M_2\tau_2(t)}{k_2^i+m_1}\bigg)\leq v(t)exp\bigg(\dfrac{c_2^s M_2\tau_2^s}{k_2^i+m_1}\bigg).
\end{equation}
On the other hand, from the $2^{sd}$ equation of  \eqref{prey-predatorII}, one has
\begin{equation}\label{ddv}
\dfrac{dv(t)}{dt}\geq v(t)\bigg[a_2^i-\dfrac{c_2^sexp\bigg(\dfrac{c_2^s M_2\tau_2^s}{k_2^i+m_1}\bigg)}{m_1+k_2^i}v(t)\bigg].\end{equation}
which yields 
$$\liminf_{t\rightarrow +\infty}v(t)\geq \dfrac{a_2^i(m_1+k_2^i)}{c_2^sexp\bigg({\dfrac{c_2^s M_2\tau_2^s}{k_2^i+m_1}}\bigg)}:=m_2.$$
\end{proof}

\textbf{Remark:} If the condition (C0) is not satisfied, every solution $(u,v)$ of system \eqref{prey-predatorII} is permanent. So, at this case we consider $m_1=0.$

\section{Existence  the p.a.p solution}
\hspace{1cm} We will offer here adequate criteria that assure the existence  of the pseudo almost periodic solution of \eqref{prey-predatorII}, as stated in the introduction. The following lemmas will be stated in order to show this conclusion.
\begin{lemma} (Theorem 2.17,\cite{cieutat2010composition}).\label{NymetskiiPAP} If $f\in PAP_U(\R\times  \X;\X)$ and for each bounded subset B of $\X$, f is bounded
on $\R\times B$, then the Nymetskii operator
$$N_f : PAP(\R,\X) )\rightarrow PAP(\R,\X)) \text{ with }N_f (u) = f(.; u(.))$$
is well defined.
\end{lemma}
\begin{lemma}\label{composepap}
 Let $\psi,\tau \in PAP(\R,\R))$.  Then $\psi(.-\tau(.)) \in PAP(\R,\R))$. 
\end{lemma}
\begin{proof}
let  $h(t,z)=\psi(t-z)$,  since the numerical application $\psi$ is continuous function  and the space $PAP(\R,\R)$ is a translation invariant then \begin{enumerate}
\item[i.] for any $z\in \R$, the function $h(.,z) \in PAP(\R? \R)$. 
\item[ ii.] for all compact K of $\R$, $\forall \epsilon> 0, \exists \delta > 0, \forall t \in \mathbb{R}, \forall z_1, z_2 \in K$,\end{enumerate}
\begin{center}
$|z_1 - z_2| \leq \delta \Rightarrow  |h(t, z_1) - h(t, z_2)| \leq \epsilon$.
\end{center}
Furthermore $\psi$ is bounded $(PAP(\R,\R)\subset BC(\R,\R))$, then $h$ is bounded on $\R\times B $ where $B$ is bounded interval. 
By the lemma \ref{NymetskiiPAP}, the Nymetskii operator
\begin{eqnarray*}
 N_f : PAP(\R,\R)&\longrightarrow & PAP(\R,\R)
 \\\tau_i &\longmapsto &h(.,\tau(.))\end{eqnarray*}
is well defined for ${\displaystyle \tau \in  PAP(\R,\mathbb{R})}$. Consequently, ${\displaystyle \bigg[t\longmapsto\psi(t-\tau(t))\bigg]\in  PAP(\R,\mathbb{R})}.$
\end{proof}
\begin{remark}
In addition, many authors suppose that $\phi$ is bounded or uniformly continuous or $inf_{t\in \R}(1-\tau'(t))>0$ for given the result. So this is the first one give the proof for any $\phi,\tau \in PAP(\R,\R)$, we have $\phi(t-\tau(t))\in PAP(\R,\R)$.
\end{remark}

 \begin{lemma}\label{AA}\item
If $a,b:t\in \R\rightarrow \R$ continuous functions, then 
\begin{align}e^{-\int^t_s a(u+\alpha)du}-e^{-\int^t_s b(u)du}=\int^t_s e^{-\int^t_r a(u+\alpha)du}.e^{-\int^r_s b(u)du}(b(r)-a(r+\alpha))dr, \ \forall \alpha\in \R.
\end{align}
\end{lemma}
\begin{proof}
Let us denote by $h$ the function defined by $h(t,s)=g_a(t,s)-g_b(t,s),$ where $g_a(t,s)=e^{-\int^t_s a(u+\alpha)du}$ and $g_b(t,s)=e^{-\int^t_s b(u)du}$. Then, by partial differentiation of $h$ with respect to the first variable, we get
$$\begin{aligned}
d_th(t,s)&=-a(t+\alpha) g_a(t,s)+b(t)g_a(t,s),\\
&=-a(t+\alpha)h(t,s)+g_b(t,s)(b(t)-a(t+\alpha)).
\end{aligned}$$
Now, by multiplying by $g_a(t,r)$ and simplifying, we deduce that
$$\begin{aligned}g_b(r,s)(b(r)-a(r+\alpha))g_a(t,r)&=d_r h(r,s)g_a(t,r)+a(r+\alpha)g_a(t,r)h(r,s),\\
&=d_r h(r,s)g_a(t,r)+d_r g_a(t,r)h(r,s),\\
&=d_r(h(r,s)g_a(t,r)).
\end{aligned}$$
Thus, by integration over the interval $[s, t]$, we have that
$$h(t,s)g_a(t,t)-g_a(t,s)h(s,s)=\int^t_s g_b(r,s)g_a(t,r)(b(r)-a(r+\alpha))dr,$$
which implies the result by noticing that $g_a(t, t) = 1$ and $h(s, s) = 0$.
\end{proof}

\begin{cor}\label{corolAA}
If $a,b:t\in \R\rightarrow \R$ continuous functions, then 
\begin{equation}e^{\int^t_s a(u)du}-e^{\int^t_s b(u)du}=\int^t_s e^{ \int^t_r a(u)du}.e^{\int^r_s b(u)du}(a(r)-b(r))dr.
\end{equation}
\end{cor}

\begin{proof}\item
We just replace $a$ ($b$ resp.) with $-a_1$  ($-b_1$ resp.) in lemma \ref{AA}.
\end{proof}

\begin{lemma}\label{APinteg}
 If $a,f \in AP(\R,\R)$, then the following function
 $$F_{\infty}(t) =\displaystyle \int^t_{-\infty} e^{-\int^t_s a(u)du}f(s)\text{ ds, }F^{\infty}(t) =\displaystyle \int_t^{+\infty} e^{-\int^t_s a(u)du}f(s)\text{ ds} \in AP(\R),\quad\forall t\in \R $$
 \end{lemma}
 \begin{proof}
 For $a \in AP(\R^+)$ and $f\in AP(\R,\R)$, for $ \epsilon >0$; $\exists l_\epsilon>0$; 
$ \exists \tau \in [l_\epsilon-n, l_\epsilon+n]\text{ where } n\in \R, \text{ then }$
$$\begin{array}{lll}

 |f(t+\tau)-f(t)|\leq \epsilon \text{ and }  |a(t+\tau)-a(t)|\leq \epsilon, \text{ } \forall t\in \R.
\end{array}$$
We pose $\epsilon=\epsilon'\dfrac{2\underline{a}^2}{\|f\|_{\infty}+\underline{a}}$ where $\epsilon'>0$.
Therefore, applying the lemma \ref{AA}, we get 
$$\begin{aligned}
\displaystyle{|F_\infty(t+\tau)-F_\infty(t)|}
=&\bigg|\displaystyle{\int^{t}_{-\infty} e^{-\int^{t}_{s}a(u-\tau)du} f(s-\tau) ds-\int^t_{-\infty}e ^{-\int^t_s a(u)du}f(s)ds\bigg|}\\
=&\bigg|\displaystyle{\int^{t}_{-\infty} e^{-\int^{t}_{s}a(u-\tau)du} f(s-\tau) ds-\int^t_{-\infty}e ^{-\int^t_s a(u)du}f(s)ds}\\
&\displaystyle{+\int^{t}_{-\infty} e^{-\int^{t}_{s}a(u)du} f(s-\tau) ds-\int^{t}_{-\infty} e^{-\int^{t}_{s}a(u)du} f(s-\tau) ds\bigg|}\\
\leq&\displaystyle{\int^t_{-\infty}\bigg|e^{- \int^t_s a(u-\tau)du}-e ^{-\int^t_s a(u)du}\bigg|ds \|f\|_{\infty}}\displaystyle{+\int^t_{-\infty} e^{-2(t-s)\underline{a}}|f(s-\tau)-f(s)|ds}\\
\leq& \displaystyle{ \|f\|_{\infty}\int^t_{-\infty}\int^t_s e^{-\int^t_r a(u-\tau)du}
e^{-\int^r_s a(u)du}|a(r-\tau)-a(r)|drds}+\dfrac{\epsilon}{2\underline{a}} \\
\leq& \displaystyle{\epsilon \|f\|_{\infty}\int^t_{-\infty} e^{-2\underline{a}(t-s)}(t-s)ds}+\dfrac{\epsilon}{2\underline{a}},\\
\leq& \displaystyle{ \|f\|_{\infty}\dfrac{\epsilon }{2\underline{a}^2}+\dfrac{\epsilon}{2\underline{a}}}=\epsilon'.
\end{aligned}
$$
and,
here, we pose $\epsilon=\epsilon''\dfrac{2\overline{a}^2}{\|f\|_{\infty}+\overline{a}}$ where $\epsilon''>0$.
Therefore, applying the corollary \ref{corolAA}, we get 
$$\begin{aligned}
\displaystyle{|F^\infty(t+\tau)-F^\infty(t)|}
\leq \displaystyle{ \|f\|_{\infty}\dfrac{\epsilon }{2\overline{a}^2}+\dfrac{\epsilon}{2\overline{a}}}=\epsilon''.
\end{aligned}
$$
 \end{proof}
Define the non-linear operator $\Upsilon$  as follows, for each $(\phi,\psi)\in PAP(\R,\R)\times PAP(\R,\R)$,\\ $\Upsilon(\phi,\psi)=(\Upsilon_1(\phi,\psi),\Upsilon_2(\phi,\psi))$ where {\small
\[ \left (\begin{aligned}\Upsilon_1(\phi,\psi)(t)\\\Upsilon_2(\phi,\psi)(t)\end{aligned}
\right )=
\left (\begin{aligned}&\int^{+\infty}_te^{\int^t_s a_1(r)dr}\bigg[b_1(s)\phi^2(s)+\dfrac{c_1(s)\psi(s-\tau_1(s))\phi(s)}{\phi(s-\sigma_1(s))+k_1(s)}\bigg]ds\\&\int^{+\infty}_t e^{\int^t_s a_2(r)dr}\dfrac{c_2(s)\psi(s-\tau_2(s))\psi(s)}{\phi(s-\sigma_2(s))+k_2(s)}ds
\end{aligned}
\right)
\]}
and Pose
{\small
\[ \left (\begin{aligned}
f_1\big(\phi(t),\phi(t-\sigma_1(t)),\psi(t-\tau_1(t))\big)\\
f_2\big(\phi(t-\sigma_2(t)),\psi(t),\psi(t-\tau_2(t))\big)
\end{aligned}
\right )=
\left (\begin{aligned}
&b_1(t)\phi^2(t)+\dfrac{c_1(t)\psi(t-\tau_1(t))\phi(t)}{\phi(t-\sigma_1(t))+k_1(t)}\\
&\dfrac{c_2(t)\psi(t-\tau_2(t))\psi(t)}{\phi(t-\sigma_2(t))+k_2(t)}.\end{aligned}
\right)
\]}

\begin{lemma}\label{PAP2}\item
If all the functions $a_{1,2},b,c_{1,2},k_{1,2},\tau_{1,2}$ and $\sigma_{1,2}$ are p.a.p and positives,  Then $\Upsilon$ maps $PAP^2(\R,\R_+)$ into itself.
\end{lemma}

\begin{proof}
Fix $\phi,\psi\in PAP(\R,\R_+)$. It follows from Lemma \ref{composepap} that $\phi(.-\sigma_{1}(.)),\psi(.-\tau_1(.))\in  PAP(\R,\R_+)$.\\ From $inf_{s\in \R}\phi(s-\sigma_1(s))+k_{1}(s)>0$ and from proprieties of space $PAP(\R,\R)$, we infer that
{\small $$f_1\big(\phi(s),\phi(s-\sigma_1(s)),\psi(s-\tau_1(s))\big)=b_1(s)\phi^2(s)+\dfrac{c_1(s)\psi(s-\tau_1(s))\phi(s)}{\phi(s-\sigma_1(s))+k_1(s)}\in PAP(\R_+).$$}
Similarly stages,   {\small $$f_2\big(\phi(s-\sigma_2(s)),\psi(s),\psi(s-\tau_2(s))\big)=\dfrac{c_2(s)\psi(s-\tau_2(s))\psi(s)}{\phi(s-\sigma_2(s))+k_2(s)}\in PAP(\R,\R_+).$$}
Consequently, for all $1 \leq j \leq 2,$ $f_j$ can be expressed as
$$f_j = f_{j1} + f_{j2}.$$
where $ f_{j1} \in AP(\R,\R_+)$ and $f_{j2} \in PAP_0(\R,\R,\R_+)$. So,
{\small
\begin{align*}
\Upsilon_j(\phi,\psi)(t)&=\displaystyle{\int^{+\infty}_t e^{\int^t_s a_j(r)dr} f_j\big(\phi(s),\phi(s-\sigma_j(s)),\psi(s-\tau_j(s))\big)ds,}\\
&\displaystyle{=\int^{+\infty}_t e^{\int^t_s a_j(r)dr}\bigg( f_{j1}\big(\phi(s),\phi(s-\sigma_j(s)),\psi(s-\tau_j(s))\big)+ f_{j2}\big(\phi(s),\phi(s-\sigma_j(s)),\psi(s-\tau_j(s))\big)\bigg)ds}\\
&\displaystyle{=I_j(t)+II_j(t)}.
\end{align*} }

By lemma \ref{APinteg}, $I_j(t)\in AP(\R,\R^+)$, for $j=1,2$.\\
In the other hand, we prove that $II_j(t)\in PAP_0(\R,\R_+)$, for j = 1,2. We have
$$\displaystyle{\int^T_{-T}\big|II_j(t)\big|dt=\int^T_{-T}\big|\int_t^{+\infty} e^{\int^t_s a_j(r)dr} f_{j2}\big(\phi(s),\phi(s-\sigma_j(s)),\psi(s-\tau_j(s))\big)ds\big|dt.}$$
Pose $f_{j2}(s)= f_{j2}\big(\phi(s),\phi(s-\sigma_j(s)),\psi(s-\tau_j(s))\big)$ and $\vartheta=t-s$. Then by Fubini Tonnelli's Theorem one has 
$$\begin{array}{lll}
\displaystyle{\dfrac{1}{2T}\int^T_{-T}\big|II_j(t)\big|dt}&\displaystyle{=\dfrac{1}{2T}\int^T_{-T}\int_t^{+\infty} e^{\int^t_s a_j(r)dr} f_{j2}(s)dsdt,}\\
&\displaystyle{\leq \dfrac{1}{2T}\int^0_{-\infty}\int^T_{-T} e^{-a_j^i\vartheta} f_{j2}(t-\vartheta)dtd\vartheta,}\\
\end{array}$$
Pose $s=t-\vartheta$, then
$$\begin{array}{lll}
\displaystyle{\dfrac{1}{2T}\int^T_{-T}\big|II_j(t)\big|dt}\displaystyle{\leq \int^0_{-\infty}\dfrac{1}{2T}\int^{T+\vartheta}_{-(T+\vartheta)} e^{-a_j^i\vartheta}f_{j2}(s)dsd\vartheta,}
\end{array}$$

Since the function $f_{j2} \in PAP_0(\R,\R^+)$, then the following function $\Psi_T$  
$$\Psi_T(\vartheta)=\dfrac{T+\vartheta}{T} \dfrac{1}{2(T+\vartheta)}\int^{T+\vartheta}_{-(T+\vartheta)}f_{j2} (s)ds$$
is bounded and $lim_{T\rightarrow+\infty}\Psi_T(\vartheta)=0$. Hence from dominated convergence
Theorem, we obtain that  $II_j(t)\in PAP_0(\R,\R_+)$, for $j = 1,2$. So for all $j=1,2$, $\Upsilon_j(\phi,\psi)$ belongs to $PAP(\R,\R_+)$ and consequently $\Upsilon$ belongs to $PAP^2(\R,\R_+)$.

\end{proof}
\begin{definition} \cite{coppel2006dichotomies}Let $X\in \R^n$ and let $A(t)$ be a $n \times n$ continuous matrix defined on
$\R$. The linear system
\begin{equation}\label{eqexpo}
X'(t) = A(t)X(t) 
\end{equation}
is said to admit an exponential dichotomy on $\R$ if there exist positive constants $k$; $h_1$; $h_2$ a projection $P$ (i.e $P^2=I$), and the fundamental solution matrix $Y (t)$ of \eqref{eqexpo} satisfying
\begin{eqnarray*}
\|Y(t)PY^{-1}(s)\|&\leq& k e^{-h_1(t-s)},\quad t\geq s\\
\|Y(t)(I-P)Y^{-1}(s)\|&\leq& k e^{-h_2(s-t)},\quad  s\geq t
\end{eqnarray*}
where $I$ is the identity matrix.
\end{definition}
\begin{lemma}\cite{coppel2006dichotomies} \label{exdic}If the linear system \eqref{eqexpo} admits an exponential dichotomy and $f \in
BC(\R,\R^n)$, then the system
\begin{equation}
X'(t) = A(t)X(t) + f(t)
\end{equation}
has a bounded solution $\overset{\backsim}{X}(t)$, and
\begin{equation}
\displaystyle \overset{\backsim}{X}(t) = \int^t_{-\infty}Y(t)PY^{-1}(s)f(s)ds -\int_t^{+\infty}Y(t)(I - P)Y^{-1}(s)f(s)ds; \end{equation}
where $Y(t)$ is the fundamental solution matrix of \eqref{eqexpo}.
\end{lemma}

\begin{theorem}
( Schauder Theorem )\cite{smart1980fixed}: \\
Let M be a non-empty convex subset of a normed space B. Let T be
a continuous mapping of M into a compact set $H \subset M$. Then T has at least fixed point.
\end{theorem}

Let define the  following set 
$$\mathcal{ M}=\{\phi,\psi\in PAP(\R);m_1\leq\phi\leq M_1 \text{ and }m_2\leq\psi\leq M_2\}.$$

\begin{theorem}\label{existenceschauder}The  differential system  \eqref{prey-predatorII} has at least  pap solution in  $M$.
\end{theorem}

\begin{proof}
  \quad Fix $\phi,\psi\in \mathcal{M}$. Let us consider the following differential system
  {\small \begin{align}\label{systemf}
\left \{\begin{aligned}
  u'(t)&=a_1(t)u(t)-b(t)\phi^2(t)-\dfrac{c_1(t)\phi(t)\psi(t-\tau_1(t))}{k_1(t)+\phi(t-\sigma_1(t))},\\
   v'(t)&=a_2(t)v(t)-\dfrac{c_2(t)\psi(t)\psi(t-\tau_2(t))}{k_2(t)+\phi(t-\sigma_2(t))}.
  \end{aligned}\right. 
\end{align} }
 It is clearly   that   {\small $$b(t)\phi^2(t)+\dfrac{c_1(t)\phi(t)\psi(t-\tau_1(t))}{k_1(t)+\phi(t-\sigma_1(t))}, \dfrac{c_2(t)\psi(t)\psi(t-\tau_2(t))}{k_2(t)+\phi(t-\sigma_2(t))} \in BC(\R,\R).$$ }
By Lemma \ref{exdic}, the system  \eqref{systemf} has a unique bounded solution given by
\[ \overset{\backsim}{X}(t)=\left (\begin{matrix}
\displaystyle \int^{+\infty}_{t} e^{\int^t_s a_1(u)du}f_1\big(\phi(s),\phi(s-\sigma_1(s)),\psi(s-\tau_1(s))\big)ds\\
\displaystyle \int^{+\infty}_{t} e^{\int^t_s a_2(u)du}f_2\big(\phi(s),\phi(s-\sigma_2(s)),\psi(s-\tau_2(s))\big)ds
\end{matrix}
\right)\]

where{\small
\[ \left (\begin{aligned}
f_1\big(\phi(t),\phi(t-\sigma_1(t)),\psi(t-\tau_1(t))\big)\\
f_2\big(\phi(t-\sigma_2(t)),\psi(t),\psi(t-\tau_2(t))\big)
\end{aligned}
\right )=
\left (\begin{aligned}
&b_1(t)\phi^2(t)+\frac{c_1(t)\psi(t-\tau_1(t))\phi(t)}{\phi(t-\sigma_1(t))+k_1(t)}\\
&\frac{c_2(t)\psi(t-\tau_2(t))\psi(t)}{\phi(t-\sigma_2(t))+k_2(t)}\end{aligned}
\right).
\]}
Then, we consider the operator $\Upsilon$. 
By lemma \ref{PAP2}, we get $\Upsilon $ maps  $PAP(\R,\R^+)\times PAP(\R,\R^+)$ into itself. By the way, ones has the following inequality 
$$u'(t)\leq a^s_1 u(t)-b^i \phi^2(t).$$
By the same lemma \ref{exdic} and by definition of $M_1$, the equation $z(t)=a^s_1 z(t)-b^i \phi^2(t)$ has a unique bounded solution given by 
$$\overset{\backsim}z(t)=\int^{+\infty}_t e^{a^s_1(t-s)} b^i \phi^2(s)ds \text{ and }\overset{\backsim}z(t)\leq M_1 .$$
Using the comparison theorem, we get 
$$\Upsilon_1(\phi,\psi)(t)\leq \overset{\backsim}z(t)\leq M_1.$$
We consider the inequality \eqref{d2v}
$$v'(t)\leq  a_2^s v(t)-\dfrac{c_2^iexp (-a_2^s \tau_2^s)}{M_1+k_2^s}\psi^2(t).$$
By the same lemma \ref{exdic} and by definition of $M_2$, the equation $z'(t)=a_2^s z(t)-\frac{c_2^iexp(-a_2^s \tau_2^s)}{M_1+k_2^s}\psi^2(t)$ has a unique bounded solution given by 
$$\overset{\backsim}z(t)=\int^{+\infty}_t e^{a^s_2(t-s)}\frac{c_2^iexp(-a_2^s \tau_2^s)}{M_1+k_2^s}\psi^2(s)ds \text{ and }\overset{\backsim}z(t)\leq M_2 .$$
Using the comparison theorem, we get 
$$\Upsilon_2(\phi,\psi)(t)\leq \overset{\backsim}z(t)\leq M_2.$$
 From the inequality \eqref{d2u}, ones has
$$
\dfrac{du(t)}{dt}\geq  a_1^iu(t)-\frac{M_2c_1^s}{k_1^i}\phi(t)-b_1^s\phi^2(t).$$
By the same lemma \ref{exdic} and by definition of $m_1$, the equation $z'(t)=a_1^iz(t)-\frac{M_2c_1^s}{k_1^i}\phi(t)-b_1^s\phi^2(t)$ has a unique bounded solution given by 
$$\overset{\backsim}z(t)=\int^{+\infty}_t e^{a^i_1(t-s)}\left(\frac{M_2c_1^s}{k_1^i}\phi(t)+b_1^s\phi^2(s)\right)ds \text{ and }\overset{\backsim}z(t)\geq m_1 .$$
Using the comparison theorem, we get 
$$\Upsilon_1(\phi,\psi)(t)\geq \overset{\backsim}z(t)\geq m_1.$$

 From the inequality \eqref{ddv}, ones has
$$\dfrac{dv(t)}{dt}\geq a_2^iv(t)-\frac{c_2^sexp\left(\frac{c_2^s M_2\tau_2^s}{k_2^i+m_1}\right)}{m_1+k_2^i}\psi^2(t).$$
Using the  lemma \ref{exdic} and by definition of $m_2$, the equation $z(t)=a_2^iz(t)-\frac{c_2^sexp\left(\frac{c_2^s M_2\tau_2^s}{k_2^i+m_1}\right)}{m_1+k_2^i}\psi^2(t)$ has a unique bounded solution given by 
$$\overset{\backsim}z(t)=\int^{+\infty}_t e^{a^i_2(t-s)}\frac{c_2^s exp\left(\frac{c_2^s M_2\tau_2^s}{k_2^i+m_1}\right)}{m_1+k_2^i}\psi^2(s)ds \text{ and }\overset{\backsim}z(t)\geq m_2 .$$
Using the comparison theorem, we get 
$$\Upsilon_2(\phi,\psi)(t)\geq \overset{\backsim}z(t)\geq m_2.$$
  Therefore, $\Upsilon(\mathcal{M})\subseteq \mathcal{M}$.
  Next step, we prove that $\Upsilon$ is continuous. For $(\phi_1,\psi_1), (\phi_2,\psi_2)\in \mathcal{M}$, such that $\|\phi_1-\phi_2\|_\infty\leq\dfrac{min(a^i_1,a^i_2)}{6\alpha} \epsilon$ and  $\|\psi_1-\psi_2\|_\infty\leq \dfrac{min(a^i_1,a^i_2)}{2\beta} \epsilon$, where
 $$\alpha=max\left(2.M_1 b^s,\frac{c_1^s M_1 M_2}{(k^i_1)^2},\frac{c_2^sM_2^2}{(k^i_2)^2},\frac{c_1^sM_2}{k^i_1}\right) \text{ and }\beta=max\left( \frac{c_1^sM_1}{k^i_1}, \frac{2c_2^s M_1}{k^i_2}\right),$$
one has
 {\small\begin{align*}
  |\Upsilon_1(\phi_1,\psi_1)(t)-\Upsilon_1(\phi_2,\psi_2)(t)|=& \displaystyle{\bigg|\int_t^{+\infty} e^{\int^t_s a_1(r)dr}f_1\left(\phi_1(s),\phi_1(s-\sigma_1(s)),\psi_1(s-\tau_1(s)))\right)\text{ } ds}\\
  &\displaystyle{-\int_t^{+\infty} e^{\int^t_s a_1(r)dr}f_1\left(\phi_2(s),\phi_2(s-\sigma_1(s)),\psi_2(s-\tau_1(s)))\right)\text{ } ds\bigg|}\\
  =&{\displaystyle{\bigg|\int_t^{+\infty} e^{\int^t_s a_1(r)dr}}}\left[b(s)\phi_1^2(s)+\frac{c_1(s)\psi_1(s-\tau_1(s))\phi_1(s)}{\phi_1(s-\sigma_1(s))+k_1(s)}\right]\text{ } ds\\
  &{\displaystyle{-\int_t^{+\infty} e^{\int^t_s a_1(r)dr}}}\left[b(s)\phi_2^2(s)+\frac{c_1(s)\psi_2(s-\tau_1(s))\phi_2(s)}{\phi_2(s-\sigma_1(s))+k_1(s)}\right]\bigg|\text{ } ds\\
   \leq&\displaystyle{\int_t^{+\infty} e^{\int^t_s a_1(r)dr}}b(s)\left|\phi_1^2(s)-\phi_2^2(s)\right|\text{ } ds\\
&{\displaystyle{+\int_t^{+\infty} e^{\int^t_s a_1(r)dr}c_1(s)}}
  \left|\frac{\psi_1(s-\tau_1(s))\phi_1(s)}{\phi_1(s-\sigma_1(s))+k_1(s)}-
\frac{\psi_2(s-\tau_1(s))\phi_2(s)}{\phi_2(s-\sigma_1(s))+k_1(s)}\right|ds
  \end{align*}}
  For simplicity, pose that 
  $${\psi_j}_s(\tau_1(s))=\psi_j(s-\tau_1(s),{\phi_j}_s(\sigma_1(s))=\phi_j(s-\sigma_1(s)),\text{ for }j=1,2.$$ 
We have{\small
\begin{align*}
 \left|\frac{{\psi_1}_s(\tau_1(s))\phi_1(s)}{{\phi_1}_s(\sigma_1(s))+k_1(s)}-
\frac{{\psi_2}_s(\tau_1(s))\phi_2(s)}{{\phi_2}_s(\sigma_1(s))+k_1(s)}\right|
&\leq\frac{{\psi_1}_s(\tau_1(s))}{{\phi_1}_s(\sigma_1(s))+k_1(s)}\|\phi_1-\phi_2\|_{\infty}+\frac{\phi_2(s)}{{\phi_1}_s(\sigma_1(s))+k_1(s)}\|\psi_1-\psi_2)\|_{\infty}\\
&+\left|\frac{{\psi_2}_s(\tau_1(s))\phi_2(s)}{{\phi_1}_s(\sigma_1(s))+k_1(s)}-\frac{{\psi_2}_s(\tau_1(s))\phi_2(s)}{{\phi_2}_s(\sigma_1(s))+k_2(s)}\right|\\
&\leq\frac{M_2}{k_1^i}\|\phi_1-\phi_2\|_{\infty}+\frac{M_1}{k_1^i}\|\psi_1-\psi_2\|_{\infty}+\frac{M_1M_2}{(k_1^i)^2}\|\phi_1-\phi_2\|_{\infty}.
\end{align*}}
And we have
\begin{align*}
|\phi_1^2(s)-\phi_2^2(s)|&= |(\phi_1(s)+\phi_2(s))(\phi_1(s)-\phi_2(s)|\\
&\leq 2M_1\|\phi_1-\phi_2\|_{\infty}.
\end{align*}
Therefore, we obtain 
 {\small \begin{align*}
  |\Upsilon_1(\phi_1,\psi_1)(t)-\Upsilon_1(\phi_2,\psi_2)(t)| \leq&\displaystyle{\int_t^{+\infty} e^{\int^t_s a_1(r)dr}\bigg\{2b^sN1 \|\phi_1-\phi_2\|_{\infty}\text{ } }\\
 &+ c_1(s)\bigg(\frac{M_2}{k_1^i}+\frac{M_1M_2}{(k_1^i)^2}\bigg)\|\phi_1-\phi_2\|_{\infty} +\frac{M_1}{k_1^i}\|\psi_1-\psi_2\|_{\infty}\bigg\}\text{ } ds\\
  \leq&{\displaystyle{\int_t^{+\infty} e^{a_1^i(t-s)}}}\bigg\{\left(2b^sM_1 +\frac{c_1^sM_2}{k_1^i}+\frac{c_1^sM_1M_2}{(k_1^i)^2}\right)\|\phi_1-\phi_2\|_{\infty}\\
&+ \frac{c_1^sM_1}{k_1^i}\|\psi_1-\psi_2\|_{\infty}\bigg\}\text{ } ds\\
   \leq&\displaystyle{\int_t^{+\infty} e^{a_1^i(t-s)}\left\lbrace 3\alpha \|\phi_1-\phi_2\|_{\infty}+\beta \|\psi_1-\psi_2\|_{\infty}\right\rbrace\text{ } ds}\\
   = &3\frac{\alpha}{a_1^i} \|\phi_1-\phi_2\|_{\infty}+\frac{\beta}{a_1^i} \|\psi_1-\psi_2\|_{\infty}\\
\leq &\frac{\epsilon}{2}+\frac{\epsilon}{2}=\epsilon.
  \end{align*}}
 By conclusion, $\Upsilon_1$ is a continuous operator.
  
  Now, we should prove that $\Upsilon_2$ is a continuous operator.
 {\small \begin{align*}
 |\Upsilon_2(\phi_1,\psi_1)(t)-\Upsilon_2(\phi_2,\psi_2)(t)|=& \displaystyle{\bigg|\int_t^{+\infty} e^{\int^t_s a_2(r)dr}f_2\left(\phi_1(s-\sigma_2(s)),\psi_1(s),\psi_1(s-\tau_2(s)))\right)\text{ } ds}\\
  &\displaystyle{-\int_t^{+\infty} e^{\int^t_s a_2(r)dr}f_2\left(\phi_2(-\sigma_2(s)s),\psi_2(s),\psi_2(s-\tau_2(s)))\right)\text{ } ds\bigg|}\\
=&\displaystyle{\bigg|\int_t^{+\infty} e^{\int^t_s a_2(r)dr}}\frac{c_2(s)\psi_1(s-\tau_2(s))\psi_1(s)}{\phi_1(s-\sigma_2(s))+k_2(s)}\text{ } ds\\
&\displaystyle{-\int_t^{+\infty} e^{\int^t_s a_2(r)dr}}\frac{c_2(s)\psi_2(s-\tau_2(s))\psi_2(s)}{\phi_2(s-\sigma_2(s))+k_2(s)}\text{ } ds\bigg|\\
\leq&{\displaystyle{\int_t^{+\infty} e^{\int^t_s a_2(r)dr}}}c_2(s)\bigg|\frac{{\psi_1}_s(\tau_2(s))\psi_1(s)}{{\phi_1}_s(\sigma_2(s))+k_2(s)}-
\frac{{\psi_2}_s(\tau_2(s))\psi_2(s)}{{\phi_2}_s(\sigma_2(s))+k_2(s)}\bigg|\text{ } ds.
  \end{align*}}
For $j=1$ or 2, pose that
 $${\psi_j}_s(\tau_2(s))=\psi_j(s-\tau_2(s)),{\phi_j}_s(\sigma_2(s))=\phi_j(s-\sigma_2(s)). $$ 
 We get
 {\small\begin{align*}
\bigg|\frac{{\psi_1}_s(\tau_2(s))\psi_1(s)}{{\phi_1}_s(\sigma_2(s))+k_2(s)}-
\frac{{\psi_2}_s(\tau_2(s))\psi_2(s)}{{\phi_2}_s(\sigma_2(s))+k_2(s)}\bigg|\leq&\frac{{\psi_1}_s(\tau_2(s))}{{\phi_1}_s(\sigma_2(s))+k_2(s)}\|\psi_1-\psi_2\|_{\infty}+\frac{\psi_2(s)}{{\phi_1}_s(\sigma_2(s))+k_2(s)}\|\psi_1-\psi_2\|_{\infty}\\
&+\frac{{\psi_1}_s(\tau_2(s))\psi_1(s)}{({\phi_1}_s(\sigma_2(s))+k_2(s))({\phi_2}_s(\sigma_2(s))+k_2(s))}\|\phi_1-\phi_2\|_{\infty}\\
\leq& \frac{2M_2}{k_2^i}\|\psi_1-\psi_2\|_{\infty}+\frac{(M_2)^2}{(k_2^i)^2}\|\phi_1-\phi_2\|_{\infty}.
\end{align*}}

Thus, we get 
{\small\begin{align*}
  |\Upsilon_2(\phi_1,\psi_1)(t)-\Upsilon_2(\phi_2,\psi_2)(t)| \leq&{\displaystyle{\int_t^{+\infty} e^{\int^t_s a_2(r)dr}c_2(s)}}\big\{ \frac{2M_2}{k_2^i}\|\psi_1-\psi_2\|_{\infty}+\frac{M_2^2}{(k_2^i)^2}\|\phi_1-\phi_2\|_{\infty}\big\}\text{ }ds\\
  \leq&{\displaystyle{\int_t^{+\infty} e^{a_2^i (t-s)}}}\big\{ \frac{2c_2^sM_2}{k_2^i}\|\psi_1-\psi_2\|_{\infty}+\frac{c_2^sM_2^2}{(k_2^i)^2}\|\phi_1-\phi_2\|_{\infty}\big\}\text{ }ds\\
  =&\frac{\beta}{a_2^i}\|\psi_1-\psi_2\|_{\infty}+\frac{\alpha}{a_2^i}\|\phi_1-\phi_2\|_{\infty} \\
  \leq& \epsilon.
  \end{align*}}
   By conclusion, $\Upsilon_2$ is a continuous operator.
  \\
  
Next, we need to prove that $\Upsilon$ is a compact operator.
In fact,  we show that the following two statements are true : \begin{enumerate}
\item[1)]   $\{\Upsilon(\phi,\psi):  (\phi,\psi)\in M\}\subset BC(\R,\R^2)$ is equi-continuous.
\item[2)]$\{\Upsilon(\phi,\psi)(t); (\phi,\psi)\in M\}$ is relatively compact subset of $\R^2$ for each $t\in \R$.\end{enumerate}

To prove 1), given $u,v\in \R$, such that $u<v$ and $|u-v| < \epsilon$, we have

\begin{align*}
\bigg|\Upsilon_1(\phi,\psi)(u)-\Upsilon_1(\phi,\psi)(v)\bigg|=&\bigg|{\displaystyle{\int_u^{+\infty} e^{\int^u_s a_1(r)dr}}}\left[b(s) \phi^2(s)+\frac{c_1(s)\psi(s-\tau_1(s))\psi(s)}{\phi(s-\sigma_1(s))+k_1(s)}\right]\text{ } ds\\
&{\displaystyle{-\int_v^{+\infty} e^{\int^v_s a_1(r)dr}}}\left[b(s) \phi^2(s)+\frac{c_1(s)\psi(s-\tau_1(s))\psi(s)}{\phi(s-\sigma_1(s))+k_1(s)}\right]\text{ } ds\bigg|\\
 =&\bigg|{\displaystyle{\int_u^{+\infty} e^{\int^u_s a_1(r)dr}}}\left[b(s) \phi^2(s)+\frac{c_1(s)\psi(s-\tau_1(s))\psi(s)}{\phi(s-\sigma_1(s))+k_1(s)}\right]\text{ } ds\\
&{\displaystyle{-\int_v^{u} e^{\int^v_s a_1(r)dr}}}\left[b(s) \phi^2(s)+\frac{c_1(s)\psi(s-\tau_1(s))\psi(s)}{\phi(s-\sigma_1(s))+k_1(s)}\right]\text{ } ds\\
 &{\displaystyle{-e^{\int^v_u a_1(r)dr}\int_u^{+\infty} e^{\int^u_s a_1(r)dr}}}\bigg[b(s) \phi^2(s)+\frac{c_1(s)\psi(s-\tau_1(s))\psi(s)}{\phi(s-\sigma_1(s))+k_1(s)}\bigg]\text{ } ds\bigg|\\
 \leq&{\bigg|1-e^{\int^v_u a_1(r)dr}\bigg|\displaystyle{\int_u^{+\infty} e^{\int^u_s a_1(r)dr}}}\left[b(s) \phi^2(s)+\frac{c_1(s)\psi(s-\tau_1(s))\psi(s)}{\phi(s-\sigma_1(s))+k_1(s)}\right]\text{ } ds\\
&{\displaystyle{+\int_v^{u} e^{\int^v_s a_1(r)dr}}}\left[b(s) \phi^2(s)+\frac{c_1(s)\psi(s-\tau_1(s))\psi(s)}{\phi(s-\sigma_1(s))+k_1(s)}\right]\text{ } ds
\end{align*}

Or $\int^v_u a_1(r)dr>0 \Rightarrow e^{\int^v_u a_1(r)dr}>1$. Therefore

\begin{align*}
\bigg|\Upsilon_1(\phi,\psi)(u)-\Upsilon_1(\phi,\psi)(v)\bigg|\leq&\left(e^{\int^v_u a_1(r)dr}-1\right){\displaystyle{\int_u^{+\infty} e^{\int^u_s a_1(r)dr}}}\left[b(s) \phi^2(s)+\frac{c_1(s)\psi(s-\tau_1(s))\psi(s)}{\phi(s-\sigma_1(s))+k_1(s)}\right]\text{ } ds\\
&{\displaystyle{+\int_{u}^v e^{\int^v_s a_1(r)dr}}}\left[b(s) \phi^2(s)+\frac{c_1(s)\psi(s-\tau_1(s))\psi(s)}{\phi(s-\sigma_1(s))+k_1(s)}\right]\text{ } ds\\
\leq&\left(e^{a_1^s(v-u)}-1\right){\displaystyle{\int_u^{+\infty} e^{\int^u_s a_1(r)dr}}}\left[b(s) \phi^2(s)+\frac{c_1(s)\psi(s-\tau_1(s))\psi(s)}{\phi(s-\sigma_1(s))+k_1(s)}\right]\text{ } ds\\
&{\displaystyle{+\int_{u}^v e^{\int^v_s a_1(r)dr}}}\left[b(s)\phi^2(s)+\frac{c_1(s)\psi(s-\tau_1(s))\psi(s)}{\phi(s-\sigma_1(s))+k_1(s)}\right]\text{ } ds\\
\leq&\dfrac{e^{a_1^s(v-u)}-1}{a^i_1}\left[b^s M_1^2+\frac{c_1^sM_1M_2}{k_1^s}\right]+\dfrac{e^{a_1^s(v-u)}-1}{a^i_1}\left[b^sM_1^2+\frac{c_1^sM_1M_2}{k_1^s}\right].
\end{align*}
 Or \begin{equation}\label{exp0}
 e^{a_1^s(v-u)}-1\underset{u\rightarrow v}{\rightarrow}0.
 \end{equation} Then
 $$\bigg|\Upsilon_1(\phi,\psi)(u)-\Upsilon_1(\phi,\psi)(v)\bigg|\underset{u\rightarrow v}{\rightarrow}0.$$
 And
\begin{align*}\left|\Upsilon_2(\phi,\psi)(u)-\Upsilon_2(\phi,\psi)(v)\right|=&\bigg|{\displaystyle{\int_u^{+\infty} e^{\int^u_s a_2(r)dr}}}\frac{c_2(s)\psi(s-\tau_2(s))\psi(s)}{\phi(s-\sigma_2(s))+k_2(s)}ds\\&{\displaystyle{-\int_v^{+\infty} e^{\int^v_s a_2(r)dr}}}\frac{c_2(s)\psi(s-\tau_2(s))\psi(s)}{\phi(s-\sigma_2(s))+k_2(s)}ds\bigg|\\
\leq&\left|1-e^{\int^v_u a_2(r)dr}\right|{\displaystyle{\int_u^{+\infty} e^{\int^u_s a_2(r)dr}}}\frac{c_2(s)\psi(s-\tau_1(s))\psi(s)}{\phi(s-\sigma_1(s))+k_2(s)}\text{ } ds\\
&{\displaystyle{+\int_{u}^v e^{\int^v_s a_2(r)dr}}}\frac{c_2(s)\psi(s-\tau_1(s))\psi(s)}{\phi(s-\sigma_1(s))+k_2(s)}\text{ } ds\\
\leq&2\left(e^{a_2^s(v-u)}-1\right)
\frac{c_2^sM_1M_2}{a^i_2k_2^s}.
\end{align*}
Then from \ref{exp0}
 $$\left|\Upsilon_2(\phi,\psi)(u)-\Upsilon_2(\phi,\psi)(v)\right|\underset{u\rightarrow v}{\rightarrow}0.$$
 which shows that 1) holds.\\
 \qquad For the  2) statement, given any $t\in \R$ and for any $(\phi,\psi)\in M$, we get 
  \begin{align*}
  \Upsilon_1(\phi,\psi)(t) \leq M_1\text{   And  }
 \Upsilon_2(\phi,\psi)(t) \leq M_2.
  \end{align*}
Pose now $\rho=max\left(M_1,M_2\right)$. Therefore $\Upsilon(\phi,\psi)(t)\in B(0,\rho)$ which proves the result.\\
Denote the closed convex hull of $\Upsilon \mathcal{M}$ by $\overline{co}\Upsilon \mathcal{M}$. Since $\Upsilon \mathcal{M}\subseteq \mathcal{M}$ and $\mathcal{M}$ is closed  convex,  $\overline{co}\Upsilon \mathcal{M}\subseteq \mathcal{M}$.  Thus 
$\Upsilon (\overline{co}\Upsilon \mathcal{M})\subseteq\Upsilon \mathcal{M}\subseteq \overline{co}\Upsilon \mathcal{M}$.  It  is  easy  to verify that $\overline{co}\Upsilon \mathcal{M}$ has the properties   1) and   2). More explicitly, $\{\Upsilon(\phi,\psi)(t);(\phi,\psi)\in \overline{co}\Upsilon \mathcal{M}\}$ is relatively compact $\R^2$ for each $t\in \R$, and $\overline{co}\Upsilon \mathcal{M}\subseteq BC(\R^2)$  is
 uniformly bounded and equi-continuous. By the Arzela-Ascoli theorem \cite{diagana2013almost}, the restriction of $\overline{co}\Upsilon \mathcal{M}$ to every bounded interval $I$ of $\R$, namely $\{\Upsilon(\phi,\psi)(t);(\phi,\psi)\in \overline{co}\Upsilon \mathcal{M}\}_{t\in I}$,  is  relatively  compact  in $C(I,\R^2)$.  Thus, $\Upsilon
 :\overline{co}\Upsilon \mathcal{M}\rightarrow\overline{co}\Upsilon\mathcal{M}$ is  a  compact  operator. It follows from Schauder's fixed point theorem that $\Upsilon$ has a fixed point $(u^*,v^*)$ in $\mathcal{M}$. The proof is complete.
  \end{proof}
   \section{Stability of the p.a.p solution }\hspace{1cm} Before the stability theorem, we need the following lemma 
\begin{lemma}\cite{gopalsamy1992stability}Let $f$ be a non-negative function defined on $[0;+\infty[$ such that $f$ is integrable on $[0;+\infty[$ and is uniformly continuous on $[0;+\infty[$. Then $$\lim_{t\rightarrow +\infty} f(t) =0.$$\end{lemma}
   \begin{definition} If $(u^*,v^*)$ is a pseudo almost periodic solution of system \eqref{prey-predatorII}, and $(u,v)$ is an any solution of \eqref{prey-predatorII} satisfying $\lim_{t\rightarrow+\infty}\big|u(t)-u^*(t)\big| =\lim_{t\rightarrow+\infty}\big|v(t)-v^*(t)\big|= 0$, then we call this 
solution $(u^*,v^*)$ is globally attractive.
\end{definition}

  \begin{theorem}\label{stabilityprey}
  Assume that $$\liminf \alpha(t),\liminf \beta(t)>0.$$
  where $\alpha(t)$ and $ \beta(t)$ are defined in \ref{alpha} and \ref{beta}. 
The pseudo almost periodic solution $(u^*,v^*)$  is globally attractive.
  \end{theorem}

  \begin{proof}\item
  Suppose $(u, v)$ is an any solution of the system \eqref{prey-predatorII}. Define the Lyapunov functional as follows
$$W_1 (t) = \big |\ln u(t)-\ln u^*(t)\big| + \big|\ln v(t)-\ln v^*(t)\big|.$$
Pose that $w_1=u-u^*$, $w_2=v-v^*$. And define $\zeta_i^{-1}$, $\varsigma^{-1}_i$ are the inverse functions of $\varsigma_i=t-\sigma_i(t)$, $\zeta_i=t-\tau_i(t)$, respectively $i=1,2$.\\
By calculating Dini derivative of $V$ along system \eqref{prey-predatorII}, we get
\begin{equation}\label{D+W1}
\begin{aligned}
 D^+ W_1(t)=& sgn(w_1(t))\left[\frac{u'(t)}{u(t)} -\frac{u^{*'} (t)}{u^*(t)} \right]+sgn(w_2(t))\left[\frac{v'(t)}{v(t)} -\frac{v^{*'} (t)}{v^*(t)} \right]\\
=& sgn(w_1(t))\left[-b_1(t)w_1(t)-c_1(t)\left(\frac{v(\zeta_1(t))}{u(\varsigma_1(t))+k_1(t)} -\frac{v^*(\zeta_1(t))}{u^*(\varsigma_1(t)+k_1(t)} \right)\right]\\
&-sgn(w_2(t))c_2(t)\left[\frac{v(\zeta_1(t))}{u(\varsigma_2(t)+k_2(t)} -\frac{v^*(\zeta_2(t))}{u^*(\varsigma_2(t))+k_2(t)} \right]\\
=& -b_1(t)\left|w_1(t)\right|-\frac{c_1(t) sgn(w_1(t))w_2(\zeta_1(t))}{u(\varsigma_1(t))+k_1(t)}+\frac{c_1(t)sgn(w_1(t))v^*(\zeta_1(t))w_1(\varsigma_1(t))}{(u(\varsigma_1(t))+k_1(t))(u^*(\varsigma_1(t))+k_1(t))}\\
&-\frac{c_2(t) sgn(w_2(t))w_2(\zeta_2(t))}{u(\varsigma_2(t))+k_2(t)}+\frac{c_2(t) sgn(w_2(t))v^*(\zeta_2(t))w_1(\varsigma_2(t))}{(u(\varsigma_2(t))+k_2(t))(u^*(\varsigma_2(t))+k_2(t))}\\
=& -b_1(t)\left|w_1(t)\right|-c_1(t) \frac{sgn(w_1(t))w_2(t)}{u(\varsigma_1(t))+k_1(t)} +c_1(t) \frac{sgn(w_1(t))}{u(\varsigma_1(t))+k_1(t)} {\displaystyle\int^t_{\zeta_1(t)}}w'_2(s)ds\\
&+\frac{c_1(t) v^*(\zeta_1(t))sgn(w_1(t))}{(u(\varsigma_1(t))+k_1(t))(u^*(\varsigma_1(t))+k_1(t))}\left[w_1(t)- {\displaystyle \int^t_{\varsigma_1(t)}w'_1(s)}ds\right]\\
&-c_2(t) \frac{sgn(w_2(t))}{u(\varsigma_2(t))+k_2(t)}\left[w_2(t)-\displaystyle\int^t_{\zeta_2(t)}w'_2(s)ds\right]\\
&+\frac{c_2(t) v^*(\zeta_2(t))sgn(w_2(t))}{(u(\varsigma_2(t))+k_2(t))(u^*(\varsigma_2(t))+k_2(t))}\left[w_1(t)- \displaystyle \int^t_{\varsigma_2(t)}w'_1(s)ds\right]\end{aligned}
\end{equation}

Integrating both sides of $w'_1(t)$ on the interval $[\varsigma_j(t),t]$ where j=1 or =2, we have
\begin{equation}\label{w'1}\begin{aligned}
\displaystyle\int^t_{\varsigma_j(t)}w'_1(s)ds=&\displaystyle\int^t_{\varsigma_j(t)}u(s)\left[a_1(s)-b_1(s)u(s)-\dfrac{c_1(s)v(\zeta_1(s))u(s)}{u(\varsigma_1(s))+k_1(s)}\right]\\
&-u^*(s)\bigg[a_1(s)-b_1(s)u^*(s)-\dfrac{c_1(s)v^*(\zeta_1(s))u^*(s)}{u^*(\varsigma_1(s))+k_1(s)}\bigg]ds\\
=&\displaystyle\int^t_{\varsigma_j(t)}a_1(s)w_1(s)-b_1(s)(u(s)+u^*(s)w_1(s)-\dfrac{c_1(t)u(s)}{u(\varsigma_1(s))+k_1(s)}w_2(\zeta_1(s))\\
&-\dfrac{c_1(s)v^*(\zeta_1(s))}{u(\varsigma_1(s))+k_1(s)}w_1(s)+\dfrac{c_1(s)v^*(\zeta_1(s))u^*(s)}{(u(\varsigma_1(s))+k_1(s))(u^*(\varsigma_1(s))+k_1(s))}w_1(\varsigma_1(s))ds
\end{aligned}\end{equation}

and Integrating both sides of $w'_2(t)$ on the interval $[\zeta_j(t),t]$ where j=1 or =2, we have
\begin{equation}\label{w'2}\begin{aligned}
\displaystyle\int^t_{\zeta_j(t)}w'_2(s)ds=&\displaystyle\int^t_{\zeta_j(t)}v(s)\bigg[a_2(s)-\dfrac{c_2(s)v(\zeta_2(s))v(s)}{u(\varsigma_2(s))+k_2(s)}\bigg]-v^*(s)\bigg[a_2(s)-\dfrac{c_2(s)v^*(\zeta_2(s))v^*(s)}{u^*(\varsigma_2(s))+k_2(s)}\bigg]ds\\
=&\displaystyle\int^t_{\zeta_j(t)}a_2(s)w_2(s)-\dfrac{c_2(t)v(s)}{u(\varsigma_2(s))+k_2(s)}w_2(\zeta_1(s))\\&-\dfrac{c_2(s)v^*(\zeta_2(s))}{u(\varsigma_2(s))+k_2(s)}w_2(s)+\dfrac{c_2(s)v^*(\zeta_2(s))v^*(s)}{(u(\varsigma_2(s))+k_2(s))(u^*(\varsigma_2(s))+k_2(s))}w_1(\varsigma_2(s))ds
\end{aligned}\end{equation}
From theorem \ref{uniformpermanent}, we have for all solution  $(u,v)$ of system \eqref{prey-predatorII} $\exists M_1>m_1\geq 0,M_2>m_2>0$ such that 
$$m_1\leq u(t)\leq M_1;\quad m_2\leq v(t)\leq M_2\qquad\forall t\in \R.$$
Therefore, By substituting \eqref{w'1}–\eqref{w'2} into \eqref{D+W1}, we get
{\small\begin{equation}\label{D++W1}\begin{aligned}
D^+W_1(t)\leq& -\left(b_1^i-\dfrac{c_1^sM_2}{(m_1+k_1^i)^2}-\dfrac{c_2^sM_2}{(m_1+k_2^i)^2}\bigg)\big |w_1(t)\big|-\bigg( \dfrac{c_2 ^i}{M_1+k_2^s}- \dfrac{c_1^s}{m_1+k_1^i}\right)\big|w_2(t)\big|\\
&
+\dfrac{c_1^s}{m_1+k_1^i}\bigg[a_2^s+\dfrac{c_2^sM_2}{m_1+k_2^i}\bigg]\displaystyle \int^t_{\zeta_1(t)}\big|w_2(s)\big|ds+\dfrac{c_1^sc_2^sM_2}{(m_1+k_1^i)(m_1+k_2^i)}\displaystyle \int^t_{\zeta_1(t)}\big|w_2(\zeta_1(s))\big|ds\\
&
+\dfrac{c_1^sc_2^sM_2^2}{(m_1+k_1^i)(m_1+k_2^i)^2}\displaystyle \int^t_{\zeta_1(t)}\big|w_1(\varsigma_2(s))\big|ds+\displaystyle\dfrac{(c_1^s)^2M_1M_2^2}{(m_1+k_1^i)^4}\int^t_{\varsigma_1(t)}\big|w_1(\varsigma_1(s))\big|ds\\
&\displaystyle +\dfrac{(c_1^s)^2M_1M_2}{(m_1+k_1^i)^3}\int^t_{\varsigma_1(t)}\big|w_2(\zeta_1(s))\big|ds
 +\bigg[\dfrac{c_1^s  a_1^sM_2}{(m_1+k_1^i)^2}+ \dfrac{2c_1^sb_1^sM_1M_2}{(m_1+k_1^i)^2}+\dfrac{(c_1^s)^2M_2^2}{(m_1+k_1^i)^3}\bigg]\displaystyle \int^t_{\varsigma_1(t)}\big|w_1(s)\big|ds\\
&+\bigg[\dfrac{c_2^sa_2^s}{m_1+k_2^i}+ \dfrac{(c_2^s)^2M_2}{(m_1+k_2^i)^2}\bigg]\displaystyle\int^t_{\zeta_2(t)}\big|w_2(s)\big|ds+\dfrac{(c_2^s)^2M_2}{(m_1+k_2^i)^2}\int^t_{\zeta_2(t)}\big|w_2(\zeta_2(s))\big|ds\\
&
+ \bigg[\dfrac{2c_2^sb_1^sM_1M_2}{(m_1+k_2^i)^2}+\dfrac{c_2^sc_1^sM_2^2}{(m_1+k_2^i)^2(m_1+k_1^i)^2}+\dfrac{c_2^sa_1^sM_2}{(m_1+k_2^i)^2}\bigg]\int^t_{\varsigma_2(t)}\big|w_1(s)\big|ds\\
&\displaystyle+\dfrac{c_2^sM_1M_2}{(m_1+k_2^i)^2(m_2+k_1^i)}\int^t_{\varsigma_2(t)}\big|w_2(\zeta_1(s))\big|ds
+\dfrac{c_2^sc_1^sM_1M_2^2}{(m_1+k_2^i)^2(m_1+k_1^i)^2}\int^t_{\varsigma_2(t)}\big|w_1(\varsigma_1(s))\big|ds\\
&+\displaystyle\dfrac{(c_2^s)^2M_2^2}{(m_1+k_2^i)^3}\int^t_{\zeta_2(t)}\big|w_1(\varsigma_2(s))\big|ds
\end{aligned}\end{equation}}
Now, Let{\small
\begin{align*}W_2(t)=&\dfrac{c_1^s}{m_1+k_1^i}\bigg[a_2^s+\dfrac{c_2^sM_2}{m_1+k_2^i}\bigg]\displaystyle \int_t^{\zeta_1^{-1}(t)}\int^t_{\zeta_1(u)}\big|w_2(s)\big|dsdu+\dfrac{c_2^sa_1^sM_2}{(m_1+k_2^i)^2}\bigg]\int_t^{\varsigma_2^{-1}(t)}\int^t_{\varsigma_2(u)}\big|w_1(s)\big|dsdu\\
&+\dfrac{c_1^sc_2^sM_2^2}{(m_1+k_1^i)(m_1+k_2^i)^2}\displaystyle\int_t^{\zeta_1^{-1}(t)} \int^t_{\zeta_1(u)}\big|w_1(\varsigma_2(s))\big|dsdu\displaystyle +\dfrac{(c_1^s)^2M_1M_2}{(m_1+k_1^i)^3}\int_t^{\varsigma_1^{-1}(t)}\int^t_{\varsigma_1(u)}\big|w_2(\zeta_1(s))\big|dsdu\\
&+\bigg[\dfrac{c_1^s a_1^sM_2}{(m_1+k_1^i)^2}+\displaystyle \dfrac{2c_1^sb_1^sM_1M_2}{(m_1+k_1^i)^2}+\dfrac{(c_1^s)^2M_2^2}{(m_1+k_1^i)^3}\bigg]\displaystyle \int_t^{\varsigma_1^{-1}(t)}\int^t_{\varsigma_1(u)}\big|w_1(s)\big|dsdu
\\
& +\displaystyle\dfrac{(c_1^s)^2M_1M_2^2}{(m_1+k_1^i)^4}\int_t^{\varsigma_1^{-1}(t)}\int^t_{\varsigma_1(u)}\big|w_1(\varsigma_1(s))\big|dsdu+\dfrac{c_2^sa_2^s}{m_1+k_2^i}\int_t^{\zeta_2^{-1}(t)}\int^t_{\zeta_2(u)}\big|w_2(s)\big|dsdu\\
&+\displaystyle\dfrac{(c_2^s)^2M_2}{(m_1+k_2^i)^2}\int_t^{\zeta_2^{-1}(t)}\int^t_{\zeta_2(u)}\big|w_2(\zeta_2(s))\big|dsdu
+ \dfrac{(c_2^s)^2M_2}{(m_1+k_2^i)^2}\int_t^{\zeta_2^{-1}(t)}\int^t_{\zeta_2(u)}\big|w_2(s)\big|dsdu\\
&+\displaystyle\dfrac{(c_2^s)^2M_2^2}{(m_1+k_2^i)^3}\int_t^{\zeta_2^{-1}(t)}\int^t_{\zeta_2(u)}\big|w_1(\varsigma_2(s))\big|dsdu+\dfrac{c_2^sM_1M_2}{(m_1+k_2^i)^2(m_1+k_1^i)}\int_t^{\varsigma_2(t)}\int^t_{\varsigma_2(u)}\big|w_2(\zeta_1(s))\big|dsdu\\
&+\displaystyle \bigg[\dfrac{2c_2^sb_1^sM_1M_2}{(m_1+k_2^i)^2}+\dfrac{c_2^sc_1^sM_2^2}{(m_1+k_2^i)^2(m_1+k_1^i)^2}+\dfrac{c_1^sc_2^sM_2}{(m_1+k_1^i)(m_1+k_2^i)}\displaystyle \int_t^{\zeta_1^{-1}(t)}\int^t_{\zeta_1(u)}\big|w_2(\zeta_1(s))\big|dsdu\\
&\displaystyle
+\dfrac{c_2^sc_1^sM_1M_2^2}{(m_1+k_2^i)^2(m_1 +k_1^i)^2}\int_t^{\varsigma_2(t)}\int^t_{\varsigma_2(u)}\big|w_1(\varsigma_1(s))\big|dsdu
\end{align*}}
then its derivative is as follows:
{\small\begin{equation}\label{W'2}\begin{aligned}W_2'(t)=&\displaystyle\dfrac{c_1^s}{m_1+k_1^i}\bigg[a_2^s+\dfrac{c_2^sM_2}{m_1+k_2^i}\bigg] 
\bigg[\big|w_2(t)\big|(\zeta_1^{-1}(t)-t)-\int^t_{\zeta_1(t)}\big|w_2(s)\big|ds\bigg]\\
&+\displaystyle\dfrac{c_1^sc_2^sM_2}{(m_1+k_1^i)(m_1+k_2^i)} \bigg[\big|w_2(\zeta_1(t))\big|(\zeta_1^{-1}(t)-t)-\int^t_{\zeta_1(t)}\big|w_2(\zeta_1(s))\big|ds\bigg]\\
&+\displaystyle\dfrac{c_1^sc_2^sM_2^2}{(m_1+k_1^i)(m_1+k_2^i)^2}\bigg[\big|w_1(\varsigma_2(t))\big|(\zeta_1^{-1}(t)-t)-\int^t_{\zeta_1(t)}\big|w_1(\varsigma_2(s))\big|ds\bigg]\\
&\displaystyle +\dfrac{(c_1^s)^2M_1M_2}{(m_1+k_1^i)^3}\bigg[\big|w_2(\zeta_1(t))\big|(\varsigma_1^{-1}(t)-t)-\int^t_{\varsigma_1(t)}\big|w_2(\zeta_1(s))\big|ds\bigg]\\
&+\displaystyle\bigg[\dfrac{c_1^s a_1^sM_2}{(m_1+k_1^i)^2}+ \dfrac{2c_1^sb_1^sM_1M_2}{(m_1+k_1^i)^2}+\dfrac{(c_1^s)^2M_2^2}{(m_1+k_1^i)^3}\bigg]\displaystyle \bigg[\big|w_1(t)\big|(\varsigma_1^{-1}(t)-t)-\int^t_{\varsigma_1(t)}\big|w_1(s)\big|ds\bigg]
\\
&+\displaystyle\dfrac{(c_1^s)^2M_1M_2^2}{(m_1+k_1^i)^4}\bigg[\big|w_1(\varsigma_1(t))\big|(\varsigma_1^{-1}(t)-t)-\int^t_{\varsigma_1(t)}\big|w_1(\varsigma_1(s))\big|ds\bigg]\\
 &+\displaystyle\bigg[\dfrac{c_2a_2^s}{m_1+k_2^i}+\dfrac{(c_2^s)^2M_2}{(m_1+k_2^i)^2}\bigg]\bigg[\big|w_2(t)\big|(\zeta_2^{-1}(t)-t)-\int^t_{\zeta_2(t)}\big|w_2(s)\big|ds\bigg]\\
&+\displaystyle\dfrac{(c_2^s)^2M_2}{m_1+k_2^i}\bigg[\big|w_2(\zeta_2(t))\big|(\zeta_2^{-1}(t)-t)-\int^t_{\zeta_2(t)}\big|w_2(\zeta_2(s))\big|ds\bigg]
\\
&+\displaystyle\dfrac{(c_2^s)^2M_2^2}{(m_1+k_2^i)^3}\bigg[\big|w_1(\varsigma_2(t))\big|(\zeta_2^{-1}(t)-t)-\int^t_{\zeta_2(t)}\big|w_1(\varsigma_2(s))\big|ds\bigg]\\&
+\displaystyle\dfrac{c_2^sM_1M_2}{(m_1+k_2^i)^2(m_1+k_1^i)}\bigg[\big|w_2(\zeta_1(t))\big|(\varsigma_2^{-1}(t)-t)-\int^t_{\varsigma_2(t)}\big|w_2(\zeta_1(s))\big|ds\bigg]\\
&+\displaystyle \bigg[\dfrac{2c_2^sb_1^sM_1M_2}{(m_1+k_2^i)^2}+\dfrac{c_2^sc_1^sM_2^2}{(m_1+k_2^i)^2(m_1+k_1^i)^2}+\dfrac{c_2^sa_1^sM_2}{(m_1+k_2^i)^2}\bigg]\bigg[\big|w_1(t)\big|(\varsigma_2^{-1}(t)-t)-\int^t_{\varsigma_2(t)}\big|w_1(s)\big|ds\bigg]\\
&\displaystyle
+\dfrac{c_2^sc_1^sM_1M_2^2}{(m_1+k_2^i)^2(m_1+k_1^i)^2}\bigg[\big|w_1(\varsigma_1(t))\big|(\varsigma_2^{-1}(t)-t)-\int^t_{\varsigma_2(t)}\big|w_1(\varsigma_1(s))\big|ds\bigg]
\end{aligned}
\end{equation}}
Define the Lyapunov functional by $W(t) = W_1(t) + W_2(t) $, then
\begin{equation}\label{D+W}
D^+W  (t) = D^+W_1(t) + W_2'(t).\end{equation}
Substitution of \eqref{D++W1}–\eqref{W'2} into \eqref{D+W} gives:
{\small \begin{align*}
D^+W  (t)\leq& -\bigg[b_1^i-\dfrac{c_1^sM_2}{(m_1+k_1^i)^2}-\dfrac{c_2^sM_2}{(m_1+k_2^i)^2}-\dfrac{c_1^sc_2^sM_2^2}{(m_1+k_1^i)(m_1+k_2^i)^2}(\zeta_1^{-1}(t)-t)\\
&-\bigg(\dfrac{c_1^s  a_1^sM_2}{(m_1+k_1^i)^2}+ \dfrac{2c_1^sb_1^sM_1M_2}{(m_1+k_1^i)^2}+\dfrac{(c_1^s)^2M_2^2}{(m_1+k_1^i)^3}+\dfrac{(c_1^s)^2M_1M_2^2}{(m_1+k_1^i)^4}\bigg)(\varsigma_1^{-1}(t)-t)\\
&-\bigg(\dfrac{(c_2^s)^2M_2^2}{(m_1+k_2^i)^3}+\dfrac{2c_2^sb_1^sM_1M_2}{(m_1+k_2^i)^2}+\dfrac{c_2^sc_1^sM_2^2}{(m_1+k_2^i)^2(m_1+k_1^i)^2}+\dfrac{c_2^sa_1^sM_2}{(m_1+k_2^i)^2}\bigg)(\varsigma_2^{-1}(t)-t)\bigg]
\big |w_1(t)\big|\\
&-\bigg[  \dfrac{c_2 ^i}{M_1+k_2^s}- \dfrac{c_1^s}{m_1+k_1^i}-\dfrac{c_1^s}{m_1+k_1^i}\bigg(a_2^s+\dfrac{c_2^sM_2}{m_1+k_2^i}+\dfrac{c_1^sc_2^sM_2}{(m_1+k_1^i)(u^i+k_2^i)}\bigg)(\zeta_1^{-1}(t)-t)\\
&-\dfrac{(c_1^s)^2M_1M_2}{(m_1+k_1^i)^3}(\varsigma_1^{-1}(t)-t)-\bigg(\dfrac{c_2^sa_2^s}{m_1+k_2^i}+ \dfrac{(c_2^s)^2M_2}{(m_1+k_2^i)^2}+\dfrac{(c_2^s)^2M_2}{(m_1+k_2^i)^2}\bigg)(\zeta_2^{-1}(t)-t)\\
&-\bigg(\dfrac{c_2^sM_1M_2}{(m_1+k_2^i)^2(m_1+k_1^i)}+\dfrac{c_2^sc_1^sM_1M_2^2}{(m_1+k_2^i)^2(m_1+k_1^i)^2}\bigg)(\varsigma_2^{-1}(t)-t)\bigg]
\big|w_2(t)\big|
\end{align*}}
Let's denote 
\begin{equation}\label{alpha}\begin{aligned}
\alpha(t)=&b_1^i-\dfrac{c_1^sM_2}{(m_1+k_1^i)^2}-\dfrac{c_2^sM_2}{(m_1+k_2^i)^2}-\dfrac{c_1^sc_2^sM_2^2}{(m_1+k_1^i)(m_1+k_2^i)^2}(\zeta_1^{-1}(t)-t)\\
&-\bigg(\dfrac{c_1^s a_1^sM_2}{(m_1+k_1^i)^2}+ \dfrac{2c_1^sb_1^sM_1M_2}{(m_1+k_1^i)^2}+\dfrac{(c_1^s)^2M_2^2}{(m_1+k_1^i)^3}+\dfrac{(c_1^s)^2M_1M_2^2}{(m_1+k_1^i)^4}\bigg)(\varsigma_1^{-1}(t)-t)\\
&-\bigg(\dfrac{(c_2^s)^2M_2^2}{(m_1+k_2^i)^3}+\dfrac{2c_2^sb_1^sM_1M_2}{(m_1+k_2^i)^2}+\dfrac{c_2^sc_1^sM_2^2}{(m_1+k_2^i)^2(m_1+k_1^i)^2}+\dfrac{c_2^sa_1^sM_2}{(m_1+k_2^i)^2}\bigg)(\varsigma_2^{-1}(t)-t)
\end{aligned}
\end{equation}
And 
\begin{equation}
\label{beta}\begin{aligned}
\beta(t)=& \frac{c_2 ^i}{M_2+k_2^s}- \frac{c_1^s}{m_1+k_1^i}-\frac{c_1^s}{m_1+k_1^i}\left(a_2^s+\frac{c_2^sM_2}{m_1+k_2^i}+\frac{c_1^sc_2^sM_2}{(m_1+k_1^i)(m_1+k_2^i)}\right)(\zeta_1^{-1}(t)-t)\\
&-\frac{(c_1^s)^2M_1M_2}{(m_1+k_1^i)^3}(\varsigma_1^{-1}(t)-t)-\left(\frac{c_2^sa_2^s}{m_1+k_2^i}+ \frac{(c_2^s)^2M_2}{(m_1+k_2^i)^2}+\frac{(c_2^s)^2M_2}{(m_1+k_2^i)^2}\right)(\zeta_2^{-1}(t)-t)\\
&-\left(\frac{c_2^sM_1M_2}{(m_1+k_2^i)^2(m_1+k_1^i)}+\frac{c_2^sc_1^sM_1M_2^2}{(m_1+k_2^i)^2(m_1+k_1^i)^2}\right)(\varsigma_2^{-1}(t)-t)
\end{aligned}
\end{equation}
From  hypothesis of the theorem, ones have  $\alpha^i=\liminf \alpha(t)$ and $\beta^i=\liminf \beta(t)$ verified for sufficiently large T, we obtained 
$$\begin{array}{lll}
D^+W  (t)&\leq&-\alpha^i\big|w_1(t)\big|-\beta^i\big|w_2(t)\big|<0.
\end{array}$$  
which implies $W (t)$ is non-increasing on $[T, +\infty[$. An integration of above inequality from T to t yields
$$ W (t) + \alpha^i \int^T_t |u(s)- u^*(s)|ds + \beta^i \int^T_t |v(s)-v^*(s)|ds \leq W(T) < +\infty,\quad \forall t > T.$$
Then $$\limsup_{t\rightarrow +\infty} \int^T_t |u(s)- u^*(s)|ds \leq \dfrac{W(T)}{\alpha^i} < +\infty \text{ and }\limsup_{t\rightarrow +\infty} \int^T_t |v(s)-v^*(s)|ds \leq \dfrac{W(T)}{\beta^i} < +\infty.$$
Thus we have 
$$\lim_{t\rightarrow +\infty} \big|u(t)-u^*(t)\big|=\lim_{t\rightarrow +\infty} \big|v(t)-v^*(t)\big|=0.$$
 \end{proof}
 \section{Example and Stimulation}
\hspace{1cm} In order to illustrate some feature of our main results, in this section, we will
apply our main results to some special prey-predator systems and demonstrate the efficiencies of our criteria.

\subsection{Example 1:} \hspace{1cm}In this example, we consider a system without the condition (C0). Then, $m_1=0$. The system is considered
\begin{equation}
\left \{
\begin{aligned}
u'(t)=&\bigg(0.04+0.125|\cos(\sqrt{2}t)|+0.125exp(-t)-(2.6+0.5\cos(t))u(t)\\&-\dfrac{3.2 v(t-0.75)}{u(t-0.75)+17 }\bigg)u(t);\\
v'(t)=&\bigg(0.01+0.25|\sin(\sqrt{7}t)|-\dfrac{3.5v(t-0.75)}{u(t-0.75)+3.4}\bigg)v(t), 
\end{aligned}\right.
\end{equation}
By a direct calculation, ones have the following table
\begin{center}
\begin{table}[hbtp]
\caption{The ecological parameters of $u$ and $v$.}
\begin{center}
\begin{tabular}{|p{0.5cm}|p{1cm}|p{1cm}| p{1cm}|p{1cm}|p{1cm}|p{1cm}|p{1cm}|p{1cm} |p{1cm} |p{1cm} |p{1cm}|l}
\hline
 $j$ & $a_j^i$& $a_j^s$& $b^i$& $b^s$&$c_j^i$& $c_j^s$&$k_j^i$& $k_j^s$ & $\sigma_j^s$& $\tau_j^s$\\
  \hline\hline
$1$& 0.04 & 0.29& 2.6 & 3.1 &3.2 &3.2 &17 &17 & 0.75& 0.75\\
\hline
$2$& 0.01 & 0.26 &---&---&3.5 &3.5 & 3.4& 3.4& 0.75&0.75 \\
\hline
\end{tabular}
\end{center}
\end{table}
\end{center}
And $\alpha^i\simeq 2.4$; $\beta^i\simeq 0.012$; $M_1=0.7085$; $m_1=0$; $M_2=0.6506$; and $m_2=0.0829$.
The theorem \eqref{existenceschauder} is verified and the conditions of  theorem \eqref{stabilityprey} are satisfied. Therefore, there exist  at least a pseudo almost periodic which is globally attractive.
\begin{figure}[ht]
\begin{center}
\includegraphics[width=0.5\linewidth]{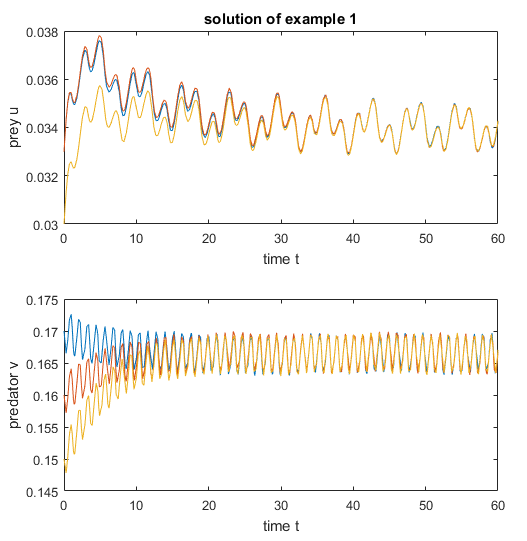}
\end{center}
\end{figure}
\begin{figure}[ht]
\begin{center}
\includegraphics[width=0.6\linewidth]{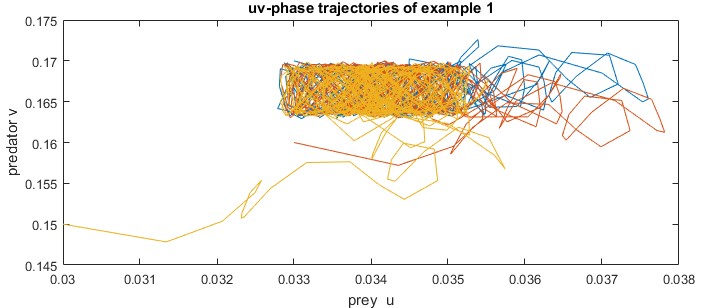}
\end{center}
\end{figure}
\subsection{Example 2:} \hspace{1cm}In this example, the condition (C0) holds. Let us consider
\begin{equation}
\left \{
\begin{aligned}
u'(t)&=\left(4.8+0.125\big(|\cos(\sqrt{2}t)|+|\cos(\sqrt{2}t)|\big)-\big(0.25|\cos(t)|\right.\\&+\left.\dfrac{33.72+32.72t^2}{4+4t^2}\big)u(t)-\dfrac{0.32 v(t-0.92)}{u(t-0.92)+16.7}\right)u(t);\\
v'(t)&=\left(0.03+0.125(|\sin(\sqrt 2t)|+|\cos(\sqrt 5t)|)-\dfrac{3.6v(t-0.92)}{u(t-0.92)+5.7}\right)v(t), 
\end{aligned}
\right. 
\end{equation}
By a direct calculation, ones have the following table
\begin{center}
\begin{table}[hptb]
\caption{The ecological parameters of $u$ and $v$.}
\begin{center}
\begin{tabular}{|p{0.5cm}|p{1cm}|p{1cm}| p{1cm}|p{1cm}|p{1cm}|p{1cm}|p{1cm}|p{1cm}|p{1cm}|p{1cm}|p{1cm}|l}
\hline
$j$ & $a_j^i$ & $a_j^s$ & $b^i$& $b^s$& $c_j^i$ & $c_j^s$ & $k_j^i$ & $k_j^s$ & $\sigma_j^s$ & $\tau_j^s$\\
  \hline\hline
$1$&4.8 & 5.05& 8.1 & 8.6 &0.32 &0.32 &16.7 &16.7 & 0.92& 0.92\\
\hline
$2$& 0.03& 0.28 &---&---&3.6 &3.6 & 5.7& 5.7& 0.92& 0.92 \\
\hline
\end{tabular}
\end{center}
\end{table}
\end{center}
And $\alpha^i\simeq 7.25$; $\beta^i\simeq 0.009$; $M_1=0.6226$; $m_1=0.5567$; $M_2=0.6403$; and $m_2=0.0408$.
The theorem \eqref{existenceschauder} is verified and the conditions of  theorem \eqref{stabilityprey} are satisfied. Therefore, there exist  at least a pseudo almost periodic which is globally attractive.
\begin{figure}[ht]
\centering
\includegraphics[width=0.5\linewidth]{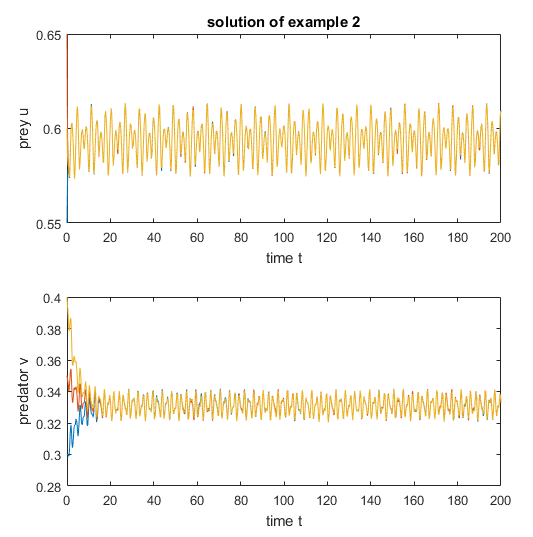}
\end{figure}

\begin{figure}[ht]
\centering
\includegraphics[width=0.6\linewidth]{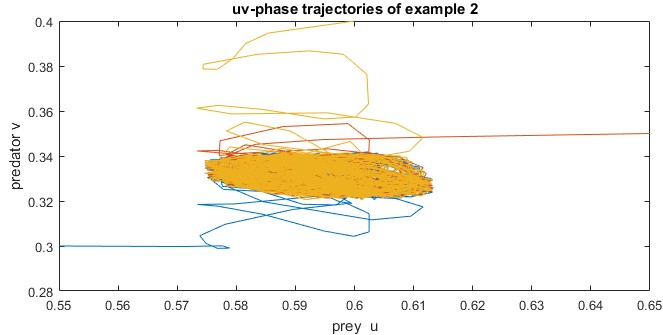}
\end{figure}
\section{Conclusion} The aim of this paper is to prove the existence of positive almost periodic solution in a Leslie-Gower predator-prey model with continuous delays. Based
on new conditions, the global attractivity of the above model is obtained by
building a suitable Lyapunov functional. Moreover, some numerical examples show  that  the our theoretical results are effective

\end{document}